\documentclass[10pt,reqno]{amsart}
\usepackage{eucal,fullpage,times,amsmath,amsthm,amssymb,mathrsfs,stmaryrd,color,enumerate,accents}
\usepackage[all]{xy}
\usepackage{url}

\newcommand{\calO}{{\mathcal{O}}}
\newcommand{\calA}{{\mathcal{A}}}

\newcommand{\Z}{\mathbf{Z}}
\newcommand{\C}{\mathbf{C}}
\newcommand{\F}{\mathbf{F}}
\newcommand{\Q}{\mathbf{Q}}
\newcommand{\V}{\mathbf{V}}
\renewcommand{\P}{\mathbf{P}}

\newcommand{\pd}{\mathrm{pd}}
\newcommand{\Spec}{{\mathrm{Spec}}}
\newcommand{\Hom}{\mathrm{Hom}}
\newcommand{\End}{\mathrm{End}}

\newcommand{\Coh}{\mathrm{Coh}}
\newcommand{\depth}{\mathrm{depth}}
\newcommand{\Vect}{\mathrm{Vect}}
\newcommand{\Frob}{\mathrm{Frob}}
\newcommand{\Mod}{\mathrm{Mod}}
\newcommand{\Ann}{\mathrm{Ann}}
\newcommand{\R}{\mathrm{R}}

\newcommand{\Pic}{\mathrm{Pic}}

\newcommand{\et}{\mathrm{\acute{e}t}}
\newcommand{\an}{\mathrm{an}}
\newcommand{\im}{\mathrm{im}}
\newcommand{\coker}{\mathrm{coker}}
\newcommand{\can}{\mathrm{can}}
\renewcommand{\ker}{\mathrm{ker}}

\newcommand{\fram}{\mathfrak{m}}
\newcommand{\fran}{\mathfrak{n}}
\newcommand{\frap}{\mathfrak{p}}

\newcommand{\colim}{\mathop{\mathrm{colim}}}

\begin{document}
\bibliographystyle{alpha}
\newtheorem{theorem}{Theorem}[section]
\newtheorem*{theorem*}{Theorem}
\newtheorem*{definition*}{Definition}
\newtheorem{proposition}[theorem]{Proposition}
\newtheorem{lemma}[theorem]{Lemma}
\newtheorem{corollary}[theorem]{Corollary}

\theoremstyle{definition}
\newtheorem{definition}[theorem]{Definition}
\newtheorem{question}[theorem]{Question}
\newtheorem{remark}[theorem]{Remark}
\newtheorem{example}[theorem]{Example}
\newtheorem{notation}[theorem]{Notation}
\newtheorem{assumption}[theorem]{Assumption}

\title{Lefschetz for Local Picard groups}
\author{Bhargav Bhatt}
\author{Aise Johan de Jong}

\begin{abstract}
We prove a strengthening of the Grothendieck-Lefschetz hyperplane theorem for local Picard groups conjectured by Koll\'ar. Our approach, which relies on acyclicity results for absolute integral closures, also leads to a restriction theorem for higher rank bundles on projective varieties in positive characteristic.
\end{abstract}

\maketitle

A classical theorem of Lefschetz asserts that non-trivial line bundles on a smooth projective variety of dimension $\geq 3$ remain non-trivial upon restriction to an ample divisor, and plays a fundamental role in understanding the topology of algebraic varieties. In \cite{SGA2}, Grothendieck recast this result in more general terms using the machinery of formal geometry and deformation theory, and also stated a local version. With a view towards moduli of higher dimensional varieties, especially the deformation theory of log canonical singularities, Koll\'ar recently conjectured \cite{KollarLefschetz} that Grothendieck's local formulation remains true under weaker hypotheses than those imposed in \cite{SGA2}. Our goal in this paper is to prove Koll\'ar's conjecture for rings containing a field.

\subsection*{Statement of results}
Let $(A,\fram)$ be an excellent normal local ring containing a field. Fix some $0 \neq f \in \fram$. Let $V = \Spec(A) - \{\fram\}$, and $V_0 = \Spec(A/f) - \{\fram\}$.  The following result is the key theorem in this paper; it solves \cite[Problem 1.3]{KollarLefschetz} completely, and \cite[Problem 1.2]{KollarLefschetz} in characteristic $0$:

\begin{theorem}
	\label{mainthm:intro}
Assume $\dim(A) \geq 4$. The restriction map $\Pic(V) \to \Pic(V_0)$ is:
\begin{enumerate}
	\item injective if $\depth_\fram(A/f) \geq 2$ and $A$ has characteristic $0$.
	\item injective up to $p^\infty$-torsion if $A$ has characteristic $p > 0$.
\end{enumerate}
\end{theorem}

This result is sharp: surjectivity fails in general, while injectivity fails in general if $\dim(A) \leq 3$, in characteristic $0$ if $\depth_\fram(A/f) < 2$, and in characteristic $p$ if one includes $p$-torsion.  A stronger similar result, including the mixed characteristic case, is due to Grothendieck \cite[Expose XI]{SGA2} under the stronger condition $\depth_\fram(A/f) \geq 3$; complex analytic variants of Grothendieck's theorem are proven in \cite{HammAnalyticLocalLefschetz}, while topological analogues are discussed in \cite{HammDungTrangTopologicalLocalLefschetz}. Without this depth constraint, a previously known case of Theorem \ref{mainthm:intro} was when $A$ has log canonical singularities $A$ in characteristic $0$, and $\{\fram\} \subset \Spec(A)$ is not an lc center (see \cite[Theorem 19]{KollarLefschetz}).

Our approach to Theorem \ref{mainthm:intro} relies on formal geometry over absolute integral closures \cite{ArtinJoin,HHBigCM}, and applies to higher rank bundles as well as projective varieties. This leads to a short proof of the following result:

\begin{theorem}
	\label{mainthmbundles:intro}
	Let $X$ be a normal projective variety of dimension $\geq 3$ over an algebraically closed field of characteristic $p > 0$. If a vector bundle $E$ on $X$ is trivial over an ample divisor, then $(\Frob^e_X)^* E \simeq \calO_X^{\oplus r}$ for $e \gg 0$.
\end{theorem}

The numerical version of Theorem \ref{mainthmbundles:intro} for line bundles is due to Kleiman \cite[Corollary 2, page 305]{KleimanAmpleness}. The non-numerical version of the rank $1$ case, with stronger assumptions on the singularities, is studied in \cite{HammDungGlobalLefschetz}.

\subsection*{An outline of the proof}
Both theorems are similar in spirit, so we only discuss Theorem \ref{mainthm:intro} here. We first prove the characteristic $p$ result, and then deduce the characteristic $0$ one by reduction modulo $p$ and an approximation argument; the reduction necessitates the (unavoidable) depth assumption in characteristic $0$.  The characteristic $p$ proof follows Grothendieck's strategy of decoupling the problem into two pieces: one in formal $f$-adic geometry, and the other an algebraisation question. Our main new idea is to replace (thanks entirely to the Hochster-Huneke vanishing theorem \cite{HHBigCM}) our ring $A$ with a very large extension $\overline{A}$ with better depth properties; Grothendieck's deformation-theoretic approach then immediately solves the formal geometry problem over $\overline{A}$. Next, we algebraise the solution over $\overline{A}$ by algebraically approximating formal sections of line bundles; the key here is to identify the cohomology of the formal completion of a scheme as the {\em derived} completion of the cohomology of the original scheme, i.e., a weak analogue of the formal functions theorem devoid of the usual finiteness constraints. Finally, we descend from $\overline{A}$ to $A$; this step is trivial in our context, but witnesses the torsion in the kernel. 

\subsection*{Acknowledgements} We thank J\'anos Koll\'ar for many helpful discussions and email exchanges. 

\section{Local Picard groups}
\label{sec:localpic}

The goal of this section is to prove Theorem \ref{mainthm:intro}. In \S \ref{sec:formalgeom}, we study formal geometry along a divisor on a (punctured) local scheme abstractly, and establish certain criteria for restriction map on Picard groups to be injective. These are applied in \S \ref{subsec:locpiccharp} to prove the characteristic $p$ part of Theorem \ref{mainthm:intro}. Using the principle of ``reduction modulo $p$'' and a standard approximation argument (sketched in \S \ref{subsec:approx}), we prove the characteristic $0$ part of Theorem \ref{mainthm:intro} in \S \ref{subsec:locpicchar0}. Finally, in \S \ref{subsec:ex}, we give examples illustrating the necessity of the assumptions in Theorem \ref{mainthm:intro}.

\subsection{Formal geometry over a punctured local scheme}
\label{sec:formalgeom}

We establish some notation that will be used in this section.

\begin{notation}
\label{not:abstractsetup}
Let $(A,\fram)$ be a local ring, and fix a regular element $f \in \fram$. Let $X = \Spec(A)$, $V = \Spec(A) - \{\fram\}$.  For an $X$-scheme $Y$,  write $Y_n$ for the reduction of $Y$ modulo $f^{n+1}$, and $\widehat{Y}$ for the formal completion of $Y$ along $Y_0$. Let $\Vect(Y)$ be the category of vector bundles (i.e., finite rank locally free sheaves) on $Y$, and write $\Pic(Y)$ and $\underline{\Pic}(Y)$ for the set and groupoid of line bundles respecively. Set $\underline{\Pic}(\widehat{Y}) := \lim \underline{\Pic}(Y_n)$ (where the limit is in the sense of groupoids), and $\Pic(\widehat{Y}) := \pi_0(\underline{\Pic}(\widehat{Y}))$. For $F \in D(\calO_Y)$, set $\widehat{F} := \R\lim (F \otimes_{\calO_Y}^L \calO_{Y_n})$;  we view $\widehat{F}$ as an $\calO_{\widehat{Y}}$-complex on $|\widehat{Y}| := Y_0$, so $\R\Gamma(\widehat{Y},\widehat{F}) := \R\Gamma(Y_0, \widehat{F}) \simeq \R\lim \R\Gamma(Y_0, F \otimes_{\calO_Y}^L \calO_{Y_n})$. The $f$-adic Tate module of an $A$-module $M$ is defined as $T_f(M) := \lim M[f^n]$; note that $T_f(M) = 0$ if $f^N \cdot M = 0$ for some $N > 0$. For any $A$-module $M$ with associated quasi-coherent sheaf $\widetilde{M}$ on $\Spec(A)$, we {\em define} $H^i_\fram(M)$ as the $i$-th cohomology of the complex $\R\Gamma_\fram(M)$ defined as the homotopy-kernel of the map $\R\Gamma(\Spec(A),\widetilde{M}) \to \R\Gamma(V,\widetilde{M})$.
\end{notation}

The following two descriptions of the cohomology of a formal completion will be crucial in this paper.

\begin{lemma}
Let $Y$ be an $X$-scheme such that $\calO_Y$ has bounded $f^\infty$-torsion. For $F \in D(\calO_Y)$, there are exact sequences
\[ 1 \to \R^1\lim H^{i-1}(Y_{n},F \otimes_{\calO_Y}^L \calO_{Y_n}) \to H^i(\widehat{Y},\widehat{F}) \to \lim H^i(Y, F \otimes_{\calO_Y}^L \calO_{Y_n}) \to 1,\]
and
\[ 1 \to \lim H^i(Y,F)/f^n \to H^i(\widehat{Y},\widehat{F}) \to T_f(H^{i+1}(Y,F)) \to 1.\]
\end{lemma}
\begin{proof}
We first give a proof when $\calO_Y$ has no $f$-torsion (which will be the only relevant case in the sequel). The first sequence is then obtained from the formula
\[ \R\Gamma(\widehat{Y},\widehat{F}) \simeq \R\lim \R\Gamma(Y,F \otimes_{\calO_Y}^L \calO_{Y_n})\]
and Milnor's exact sequence for $\R\lim$. Applying the projection formula (since $A/f^n$ is $A$-perfect) to the above gives
\[ \R\Gamma(\widehat{Y},\widehat{F}) \simeq \R\lim \big(\R\Gamma(Y,F) \otimes_A^L A/f^n). \]
The second sequence is now obtained by applying the derived $f$-adic completion functor $\R\lim(- \otimes_A^L A/f^n)$ to the canonical filtration on $\R\Gamma(Y,F)$, which proves the claim. In general, the boundedness of $f$-torsion in $\calO_Y$ shows that the map $\{\calO_Y \stackrel{f^n}{\to} \calO_Y\} \to \{\calO_{Y_n}\}$ of projective systems is a (strict) pro-isomorphism, and hence $\{F \stackrel{f^n}{\to} F\} \to \{F \otimes^L_{\calO_Y} \calO_{Y_n}\}$ is also a pro-isomorphism. Now the previous argument applies.
\end{proof}

The following conditions on the data $(A,f)$ will be assumed throughout this subsection; we do {\em not} assume $A$ is noetherian as this will not be true in applications.

\begin{assumption}
	\label{ass:formalgeomassumptions}
	Assume that the data from Notation \ref{not:abstractsetup} satisfies the following:
	\begin{itemize}
		\item $X$ is integral, i.e., $A$ is a domain.
		\item $j:V \hookrightarrow X$ is a quasi-compact open immersion, i.e., $\fram$ is the radical of a finitely generated ideal.
		\item $H^0(V,\calO_V)$ is a finite $A$-module.
		\item $f^N \cdot H^1(V,\calO_V) = 0$ for $N \gg 0$.
	\end{itemize}
\end{assumption}

\begin{example}
	Any $S_2$ noetherian local domain $(A,\fram)$ of dimension $\geq 3$ admitting a dualising complex satisfies Assumption \ref{ass:formalgeomassumptions}: the $A$-module $H^2_\fram(A) \simeq H^1(V,\calO_V)$ has finite length (see \cite[Corollary VIII.2.3]{SGA2}), while $H^0(V,\calO_V) \simeq A$ as $A$ is $S_2$. The absolute integral closure of a {\em complete} noetherian local domain  of dimension $\geq 3$ in characteristic $p$ also satisfies these conditions (see Theorem \ref{thm:hhbigcmalg}), and is a key example for the sequel.
\end{example}

We now study formal geometry over $\widehat{V}$. The following elementary bound on the $f^\infty$-torsion of certain cohomology groups will help relate sheaf theory on $\widehat{V}$ to that on $V$.

\begin{lemma}
	\label{lem:notatemod}
	For $E \in \Vect(V)$, one has $f^k \cdot H^1(V,E) = 0$ for $k \gg 0$.
\end{lemma}
\begin{proof}
Fix an $N$ with $f^N \cdot H^1(V,\calO_V) = 0$, and set $\fram' := \Ann_A(f^N \cdot H^1(V,E)) \subset \fram$. For each $\frap \in V \subset \Spec(A)$, there is a $g \in \fram - \frap$ and an isomorphism $E|_{D(g)} \simeq (\calO_V^{\oplus r})|_{D(g)}$. Clearing denominators gives an exact sequence
	\[ 1 \to \calO_V^{\oplus r} \to E \to Q \to 1\]
with $g^n \cdot Q = 0$ for some $n > 0$ (by quasi-compactness). Then $g^n \in \fram'$, so $\fram' \not\subset \frap$. Varying over all $\frap \in V$ shows that $A/\fram'$ is a local ring with a unique prime ideal $\fram/\fram'$, so $f^m \in \fram'$ for $m \gg 0$, and hence $f^{N+m} \cdot H^1(V,E) = 0$.
\end{proof}

We can now algebraically approximate formal sections of vector bundles on $V$:

\begin{lemma}
	\label{lem:sectcomplete}
For $E \in \Vect(V)$, one has $\widehat{H^0(V,E)} \simeq H^0(\widehat{V},\widehat{E})$.
\end{lemma}
\begin{proof}
	Lemma \ref{lem:notatemod} shows that $\{H^1(V,E)[f^n]\}$ is essentially $0$, so $T_f(H^1(V,E)) = 0$. It remains to observe that $\widehat{H^0(V,E)} \simeq \pi_0(\widehat{H^0(V,E)})$ since $f$ is a non-zero divisor on $H^0(V,E)$.
\end{proof}

One can also prove the following Lefschetz-type result for $\pi_1$:

\begin{corollary}
	\label{cor:lefschetzpi1}
The natural map $\pi_{1,\et}(V_0) \to \pi_{1,\et}(V)$ is surjective if $A$ is noetherian and $f$-adically complete.
\end{corollary}
\begin{proof}
We want $\pi_0(W) \simeq \pi_0(W_0)$ for any finite \'etale cover $W \to V$.  If $\calA$ is a finite flat quasi-coherent $\calO_V$-algebra,  then $H^0(V,\calA) \simeq \widehat{H^0(V,\calA)} \simeq H^0(\widehat{V},\widehat{\calA}) \simeq \lim H^0(V_n,\calA_n)$ by the noetherian assumption and Lemma \ref{lem:sectcomplete}. Hence, if $\calO_V \to \calA$ is also \'etale, then $H^0(V,\calA) \to H^0(V_n,\calA_n) \to H^0(V_0,\calA_0)$ induce bijections on idempotents.
\end{proof}

Next, we show that pullback along $\widehat{V} \to V$ is faithful on line bundles.

\begin{lemma}
	\label{lem:piccomplete}
	The natural map $\Pic(V) \to \Pic(\widehat{V})$ is injective.
\end{lemma}
\begin{proof}
Fix an $L \in \ker(\Pic(V) \to \Pic(V_0))$.  Lemma \ref{lem:sectcomplete} gives an injective map $s:L \to \calO_V $ with $s|_{V_0}$  an isomorphism.  Hence, if $Q = \coker(s)$, then multiplication by $f$ is an isomorphism on $Q$, so $H^0(V,Q)$ is uniquely $f$-divisible. Lemma \ref{lem:notatemod} shows $f^N \cdot H^1(V,L) = 0$ for $N \gg 0$, so $H^0(V,\calO_V) \to H^0(V,Q)$ is surjective, and hence  $H^0(V,Q)$ is a finitely generated $f$-divisible $A$-module. By Nakayama, $H^0(V,Q) = 0$, so $Q = 0$ as $\calO_V$ is ample.
\end{proof}

\begin{remark}
	The same argument shows $\Vect(V) \to \Vect(\widehat{V})$ is injective on isomorphism classes. If $V_0$ is $S_2$, then one can show that each $\widehat{E} \in \Vect(\widehat{V})$ algebraises to some torsion free $E \in \Coh(V)$ (see \cite[Theorem IX.2.2]{SGA2}); examples such as \cite[Example12]{KollarLefschetz} show that $E$ need not be a vector bundle, even in the rank $1$ case.
\end{remark}

The next observation is a manifestation of the formula $\widehat{V} = \colim_n V_n$ and some bookkeeping of automorphisms:

\begin{lemma}
	\label{lem:piccont}
	The natural map $\Pic(\widehat{V}) \to \lim \Pic(V_n)$ is bijective.
\end{lemma}
\begin{proof}
	Since $\underline{\Pic}(\widehat{V}) \simeq \lim \underline{\Pic}(V_n)$ as groupoids, it suffices to show $\{\pi_1(\Pic(V_n))\} := \{H^0(V_n,\calO_{V_n}^*)\}$ satisfies the Mittag-Leffler (ML) condition. The assumption on $V$ shows that $\{H^1(V,\calO_V)[f^n]\}$ is essentially $0$, and hence $\{H^0(V_n,\calO_{V_n})\}$ satisfies ML. Since $|V_0| = |V_n|$, we have 
\[ \{H^0(V_n,\calO_{V_n}^*) \}= \{H^0(V_n,\calO_{V_n}) \times_{H^0(V_0,\calO_{V_0})} H^0(V_0,\calO_{V_0}^*)\}\]
as projective systems. The claim now follows from Lemma \ref{lem:mlbasechange}.
\end{proof}

\begin{lemma}
	\label{lem:mlbasechange}
	If $\{X_n\}$ is a projective system of sets that satisfies ML, and $Y_0 \to X_0$ is some map, then the base change system $\{Y_n\} := \{Y_0 \times_{X_0} X_n\}$ also satisfies ML.
\end{lemma}
\begin{proof}
	Let $Z_{n,k} \subset X_k$ be the image of $X_n \to X_k$ for any $k \leq n$. The assumption says: for fixed $k$, one has $Z_{n,k} = Z_{n+1,k}$ for $n \gg 0$.  Since $\im(X_n \times_{X_0} Y_0 \to X_k \times_{X_0} Y_0)  = Z_{n,k} \times_{X_0} Y_0$, the claim follows.
\end{proof}

We quickly recall the standard deformation-theoretic approach to studying line bundles on $\widehat{V}$:

\begin{lemma}
	\label{lem:piconestep}
The map $\Pic(V_{n+1}) \to \Pic(V_n)$ is injective if $H^1(V_0,\calO_{V_0}) = 0$, and surjective if $H^2(V_0,\calO_{V_0}) = 0$.
\end{lemma}
\begin{proof}
	Standard using the exact sequence $1 \to \calO_{V_0} \stackrel{a}{\to} \calO_{V_{n+1}}^* \to \calO_{V_n}^* \to 1$ where $a(g) = 1 + g \cdot f^n$.
\end{proof}

We end by summarising the relevant consequences of the preceding discussion:

\begin{corollary} 
	\label{cor:resinjcrit} For $A$ satisfying Assumption \ref{ass:formalgeomassumptions}, we have:
	\begin{enumerate}
		\item The map $\Pic(V) \to \Pic(\widehat{V})$ is injective.
		\item The map $\Pic(\widehat{V}) \to \lim \Pic(V_n)$ is bijective. 
		\item The map $\Pic(\widehat{V}) \to \Pic(V_0)$ is injective if $H^1(V_0,\calO_{V_0}) = 0$.
	\end{enumerate}
\end{corollary}
\begin{proof}
We simply combine lemmas \ref{lem:piccomplete}, \ref{lem:piccont}, and \ref{lem:piconestep}.
\end{proof}

\subsection{Characteristic $p$}
\label{subsec:locpiccharp}

We follow Notation \ref{not:abstractsetup}. Our goal is to prove the following:

\begin{theorem}
	\label{mainthm:charp}
Fix an excellent normal local $\F_p$-algebra $(A,\fram)$ of dimension $\geq 4$, and some $0 \neq f \in \fram$. Then the kernel of $\Pic(V) \to \Pic(V_0)$ is $p^\infty$-torsion.
\end{theorem}

The rest of \S \ref{subsec:locpiccharp} is dedicated to proving Theorem \ref{mainthm:charp}, so we fix an $(A,\fram,f)$ as in Theorem \ref{mainthm:charp} at the outset. The first reduction is to the complete case:

\begin{lemma}
	\label{lem:reducetocomplete}
	If $\pi:\Spec(R) \to \Spec(A)$ is $\fram$-adic completion of $A$, then $\Pic(V) \to \Pic(\pi^{-1}(V))$ is injective.
\end{lemma}
\begin{proof}
A line bundle $L \in \Pic(V)$ extends to a unique finite $A$-module $M$ with $\depth_\fram(M) \geq 2$, and similarly for line bundles on $\Pic(\pi^{-1}(V))$. Since $\pi^*:\Mod^f_A \to \Mod^f_R$ preserves depth, it suffices to prove: if $M \in \Mod^f_A$ with $M \otimes_A R \simeq R$, then $M \simeq A$. For this, we simply observe that an isomorphism $R \simeq M \otimes_A R$ can be approximated modulo $\fram$ by a map $A \to M$ which is injective (since $A$ is a domain) and surjective by Nakayama, so $M \simeq A$.
\end{proof}

By Lemma \ref{lem:reducetocomplete} and the preservation of normality under completion of excellence rings, to prove Theorem \ref{mainthm:charp}, we can (and do) assume $A$ is an $\fram$-adically complete noetherian local normal ring. To proceed further, we define:

\begin{notation}
Let $\overline{A}$ denote a fixed absolute integral closure of $A$. For any $A$-scheme $Y$, we write $\overline{Y} := Y_{\overline{A}}$.
\end{notation}

Our strategy for proving Theorem \ref{mainthm:charp} is to first prove $\Pic(\overline{V}) \to \Pic(\overline{V}_0)$ is injective, and then descend to a finite level conclusion via norms. The situation over $\overline{V}$ is analysed via the formal geometry of \S \ref{sec:formalgeom}. The reason we work at the infinite level first is that formal geometry is easier over $\overline{V}$ than over $V$, thanks to the following vanishing result:

\begin{theorem}
	\label{thm:hhbigcmalg}
$\overline{A}$ is Cohen-Macaulay, i.e., $H^i_{\fram}(\overline{A}) = 0$ for $i < \dim(A)$.
\end{theorem}

\begin{remark}
	Strictly speaking, the local cohomology groups used in Theorem \ref{thm:hhbigcmalg} are defined as the derived functors of sections supported at $\{\fram\} \subset \Spec(A)$ applied to $\overline{A}$. These do not {\em a priori} agree with those arising from the definition adopted in Notation \ref{not:abstractsetup}. However, both approaches to local cohomology commute with filtered colimits. Hence, for both definitions, we have $H^i_\fram(\overline{A}) = \colim H^i_\fram(B)$ where the colimit ranges over finite extensions $A \to B$ contained in $\overline{A}$. By reduction to the noetherian case, the two definitions of $H^i_\fram(\overline{A})$ coincide.
\end{remark}

Theorem \ref{thm:hhbigcmalg} is due to Hochster-Huneke \cite{HHBigCM}, and can be found in \cite[Corollary 2.3]{HLBigCM} in the form above. It implies $H^i(\overline{V},\calO_{\overline{V}}) = 0$ for $0 < i <\dim(A) - 1$, so $H^i(\overline{V}_0,\calO_{\overline{V}_0}) = 0$ for $0 < i < \dim(A) - 2$. We use this to prove an infinite level version of Theorem \ref{mainthm:charp}:

\begin{proposition}
	\label{prop:mainthminflevel}
	The map $\Pic(\overline{V}) \to \Pic(\overline{V}_0)$ is injective if $\dim(A) \geq 4$.
\end{proposition}
\begin{proof}
This follows from Corollary \ref{cor:resinjcrit} as $\overline{A}$ satisfies the relevant conditions by Theorem \ref{thm:hhbigcmalg} since $\dim(A) \geq 4$.
\end{proof}

We can now descend down to prove the main theorem:

\begin{proof}[Proof of Theorem \ref{mainthm:charp}]
	Fix an $L \in \ker(\Pic(V) \to \Pic(V_0))$.  Proposition \ref{prop:mainthminflevel} shows $L \in \ker(\Pic(V) \to \Pic(\overline{V}))$. By expressing $\overline{A}$ as a filtered colimit of finite extensions, it follows that $L \in \ker(\Pic(V) \to \Pic(W))$ for a finite surjective map $W \to V$. As $V$ is normal, using norms (see \cite[\S XVII.6.3]{SGA4.3}), we conclude that $L$ is torsion. It now suffices to rule out the presence of prime-to-$p$ torsion in $\ker(\Pic(V) \to \Pic(V_0))$. Corollary \ref{cor:resinjcrit} shows that this kernel is contained in the kernel of $\lim \Pic(V_n) \to \Pic(V_0)$. The kernel of $\Pic(V_{n+1}) \to \Pic(V_n)$ is an $\F_p$-vector space for each $n$, so $\lim \Pic(V_n) \to \Pic(V_0)$ has no prime-to-$p$ torsion in the kernel.
\end{proof}

\begin{remark}
In the setting of Theorem \ref{mainthm:charp}, the proof above also shows: if $E \in \Vect(V)$ is trivial over $V_0$, i.e., satisfies $E|_{V_0} \simeq \calO_{V_0}^{\oplus n}$, then $E$ is trivialised by a finite extension of $V$.
\end{remark}

\subsection{Characteristic $0$}
\label{subsec:locpicchar0}

We follow Notation \ref{not:abstractsetup}. Our goal is to prove the following:

\begin{theorem}
	\label{thm:char0}
Fix an excellent normal local $\Q$-algebra $(A,\fram)$ of dimension $\geq 4$, and some $0 \neq f \in \fram$. Assume $\depth_\fram(A/f) \geq 2$. Then $\Pic(V) \to \Pic(V_0)$ is injective.
\end{theorem}

\begin{proof}
	By Lemma \ref{lem:reducethmefp} below, we may assume that $A$ is an essentially finitely presented $\Q$-algebra. The depth assumption implies that $\depth_\fram(A) \geq 3$ as $f$ acts nilpotently $H^2_\fram(A)$ with kernel $H^1_\fram(A/f) = 0$. Now fix a line bundle $L$ in the kernel of $\Pic(V) \to \Pic(V_0)$. By spreading out (see \cite[\S 2]{HochsterCharpApp}), we can find:
	\begin{enumerate}
		\item A mixed characteristic dvr $(\calO,(\pi))$ with perfect residue field of characteristic $p > 0$.
		\item A normal noetherian $\calO$-flat local ring $\widetilde{A}$ satisfying: 
			\begin{enumerate}
				\item  There is a map $\widetilde{A}[1/\pi] \to A$.
				\item $B := \widetilde{A}/\pi$ is normal of dimension $\dim(A)$ and has depth $\geq 3$ at its closed point.
			\end{enumerate}
		\item A section $\widetilde{A} \to \calO$ of the structure map $\calO \to \widetilde{A}$ defined by an ideal $\widetilde{\fram} \subset \widetilde{A}$ that, after inverting $\pi$, gives the image of the closed point under $\Spec(A) \to \Spec(\widetilde{A}[1/\pi])$.
		\item An element $\widetilde{t} \in \widetilde{A}$ such that $\widetilde{A}/\widetilde{t}$ is $\calO$-flat and maps to $t$ along $\widetilde{A} \to \widetilde{A}[1/\pi] \to A$.
		\item A line bundle $\widetilde{L}$ on $\widetilde{V}$ which induces $L$ over $V$ and lies in the kernel of $\Pic(\widetilde{V}) \to \Pic(\widetilde{V}_0)$;  here $\widetilde{V} = \Spec(\widetilde{A}) - \{ \widetilde{\fram} \}$, and the subscript $0$ denoting passage to the $\widetilde{t} = 0$ fibre.
	\end{enumerate}
	Write $U = \Spec(B) - \{\widetilde{\fram} \cdot B\}$ for the punctured spectrum of $B$, and use the subscript $0$ to indicate passage to the $\tilde{t} = 0$ fibre. Then we have a commutative diagram
	\[ \xymatrix{ \Pic(\widetilde{V}) \ar[r]^-a \ar[d]^-b & \Pic(\widetilde{V}_0) \ar[d]^-c \\
			\Pic(U) \ar[r]^-d & \Pic(U_0) } \]
			where the vertical maps are induced by reduction modulo $\pi$, while the horizontal maps are induced by reduction modulo $\widetilde{t}$. Theorem \ref{mainthm:charp} tells us that the kernel of $d$ is $p^\infty$ torsion. Corollary \ref{cor:resinjcrit} shows  $b$ is injective, so $\widetilde{L}$ (and hence $L$) is is killed by a power of $p$. Repeating the above construction by spreading out over a mixed characteristic dvr whose residue characteristic is $\ell \neq p$, it follows that $L$ is also killed by a power of $\ell$, and is hence trivial.
\end{proof}

\begin{remark}
We do not know a proof of Theorem \ref{thm:char0} that avoids reduction modulo $p$ except when $A$ is $S_3$, where one can argue directly as follows. By Lemma \ref{lem:piccomplete}, it suffices to prove $\Pic(\widehat{V}) \to \Pic(V_0)$ is injective. The kernel of this map is $H^1(\widehat{V},1 + \widehat{I})$, where $I = (f) \subset \calO_V$ is the ideal defining $V_0$. In characteristic $0$, the exponential gives an isomorphism $\widehat{I} \simeq 1 + \widehat{I}$ of sheaves on $\widehat{V}$, so it suffices to prove $H^1(\widehat{V},\widehat{I}) = 0$. Using $f:\calO_V \simeq I$  and $H^1(V,\calO_V) = 0$ (since $\depth_\fram(A) \geq 3$), it suffices to show $T_f(H^2(V,\calO_V)) = 0$. The $A$-module $H^2(V,\calO_V)$ has finite length as $A$ is $S_3$, so $T_f(H^2(V,\calO_V)) = 0$. If $\depth_\fram(A) \geq 3$ but $A$ is not $S_3$, then the last step fails; in fact, there are examples \cite[Example 12]{KollarLefschetz} of such $A$ where $\Pic(\widehat{V}) \to \Pic(V_0)$ is not injective, rendering this approach toothless in general.
\end{remark}

\subsection{An approximation argument}
\label{subsec:approx}
We now explain the approximation argument used to reduce Theorem \ref{thm:char0} to the case of essentially finitely presented algebras over $\Q$. First, we show how modules over the completion of an excellent ring can be approximated by modules over a smooth cover while preserving homological properties.

\begin{lemma}
	\label{lem:approxmod}
	Fix an excellent henselian local ring $(P,\fran)$ with $\fran$-adic completion $\widehat{P}$.  Let $I$ be the category of diagrams $P \to S \to \widehat{P}$ with $P \to S$ essentially smooth and $S$ local. Then one has
	\begin{enumerate}
		\item $I$ is filtered, and $\widehat{P} \simeq \colim_I S$.
		\item $\colim_I \Mod^f S \simeq \Mod^f \widehat{P}$ via the natural functor.
		\item If $M \in \Mod^f_{\widehat{P}}$ has $\pd_{\widehat{P}}(M) < \infty$, then there exists $S \in I$ and $N \in \Mod^f_S$ such that $N \otimes_S^L \widehat{P} \simeq M$.
	\end{enumerate}
\end{lemma}
\begin{proof}
	(1) is Popescu's theorem \cite{SwanPopescu}, while (2) is automatic from (1) as all rings in sight are noetherian. Now pick $M \in \Mod^f_{\widehat{P}}$ as in (3) with a finite free resolution $K \to M$ over $\widehat{P}$. Then there exists an $S \in I$ and a finite free $S$-complex $L$ such that $L \otimes_S \widehat{P} = K$ as complexes. It suffices to thus check that $L \in D^{\geq 0}(S)$.  Write $j:P \to S$ and $a:S \to \widehat{P}$ for the given maps.As $P$ is henselian, for each integer $c$, there exists a section $S \to P$ of $j$ such that the composite $b:S \to P \to \widehat{P}$ agrees with $a$ modulo $\fran^c$. Then \cite[Lemma 3.1]{CdJApprox} shows that  $L \otimes_{S,b} \widehat{P}$ is acyclic outside degree $0$ (for sufficiently $c$).  The same is also true for $L \otimes_{S} P$ by faithful flatness.  If $I = \ker(S \to P)$, then $I$ is a regular ideal contained in the Jacobson radical of $S$ (since $S$ is local and essentially $P$-smooth). Let $\widehat{S}$ be the $I$-adic completion of $S$, so $S \to \widehat{S}$ is faithfully flat. By the formula $L \otimes_S \widehat{S} \simeq \R\lim(L \otimes_S S/I^n)$, it suffices to show that the right hand side lies in $D^{\geq 0}(S)$.  The regularity of $I$ shows that each $I^n/I^{n+1}$ is a free $S/I$-module (as $S/I = P$ is local), so $L \otimes_S S/I^n \in D^{\geq 0}(S)$ by devissage as $L \otimes_S S/I \in D^{\geq 0}(S)$.
\end{proof}

The approximation argument used above permits us to make the promised reduction:

\begin{lemma}
	\label{lem:reducethmefp}
To prove Theorem \ref{thm:char0}, it suffices to do so when $A$ is essentially finitely presented over $\Q$.
\end{lemma}
\begin{proof}
We may assume the conclusion of Theorem \ref{thm:char0} is known all essentially finitely presented normal local $k$-algebras $A$ of depth $\geq 3$ over a characteristic $0$ field $k$ (the passage from $k = \Q$ to general $k$ is routine and left to the reader). By Lemma \ref{lem:reducetocomplete} and excellence of $A$, it suffices to show the conclusion holds for all triples $(A,\fram,f)$ where $(A,\fram)$ is a complete noetherian local normal ring with $\depth_\fram(A) \geq 3$ in characteristic $0$, and $0 \neq f \in \fram$. 

If $k = A/\fram$, then we choose a Cohen presentation $A = \widehat{P}/I$ where $P$ is the henselisation at $0$ over $k[x_1,\dots,x_n]$, and $\widehat{P}$ is the completion. Choose an element $f \in \widehat{P}$ lifting $f \in A$, and a finite $A$-module $M$ with $\depth_\fram(M) \geq 2$ corresponding to an element in the kernel of $\Pic(V) \to \Pic(V_0)$, where $V = \Spec(A) - \{\fram\}$, and $V_0 = V \cap \Spec(A/f)$. Observe that $\pd_{\widehat{P}}(A) \leq n-3$ and $\pd_{\widehat{P}}(M) \leq n-2$ by Auslander-Buschbaum. We will show $A \simeq M$. 

By Lemma \ref{lem:approxmod}, we can find:
	\begin{enumerate}
		\item A factorisation $P \stackrel{j}{\to} S \stackrel{a}{\to} \widehat{P}$ with $(S,\fran)$ a local essentially smooth $P$-algebra.
		\item A quotient $S \to B$ such that $B \otimes^L_S \widehat{P} \simeq A$.
		\item A finite $B$-module $M'$ invertible on $V_B = \Spec(B) - V(x_1,\dots,x_n)$  such that $M' \otimes_B^L A \simeq M$.
		\item A lift of $f$ to $\fran \subset S$ such that $M'$ is the trivial line bundle on $V_B \cap \Spec(B/f)$.
	\end{enumerate}
	We remark that $\pd_S(B) \leq n-3$ as $B \otimes^L_S \widehat{P} \simeq A$, and similarly $\pd_S(M') \leq n-2$. As $P$ is henselian and $S$ is $P$-smooth with a section over $\widehat{P}$, we may choose a large enough constant $c$ (depending on $M$ and $A$ as $\widehat{P}$-modules) and a section $s_c:S \to P$ of $j$ that coincides with $a$ modulo $(x_1,\dots,x_n)^c$. Set $A_c = B \otimes_S^L P$ and $M_c = M' \otimes_S^L P$. Then, by choice of $c$, both these complexes are in fact discrete, and hence $A_c$ is a local quotient of $P$. Let $\fram_c \subset A_c$ be the maximal ideal; this is the image of $\fran$, and also generated by $\{x_1,\dots,x_n\}$. We call this triple $(A_c,\fram_c,M_c)$ an approximation of $(A,\fram,M)$, and observe that better approximations can be found by replacing $c$ with a larger integer. At the expense of performing this operation, we have:
	\begin{enumerate}
		\item  $(A_c,\fram_c)$ coincides with $(A,\fram)$ modulo $(x_1,\dots,x_n)^c$, and $\dim(A_c) = \dim(A)$ as the Hilbert series of $(A_c,\fram_c)$ and $(A,\fram)$ coincide (see \cite[Theorem 3.2]{CdJApprox}).
		\item $\depth_{\fram_c}(A_c) \geq 3$, and $\depth_{\fram_c}(M_c) \geq 2$ by Auslander-Buschbaum over $P$.
		\item The singular locus of $\Spec(A_c)$ has codimension $\geq 2$ by the Jacobian criterion.
		\item $M_c$ is invertible over $U := \Spec(A_c) - V(x_1,\dots,x_n) = \Spec(A_c) -\{\fram_c\}$. 
		\item $M_c$ restricts to the trivial line bundle over $U \cap \Spec(A_c/f)$.
	\end{enumerate}
	By (2) and (3), such an $A_c$ is in particular normal. As Theorem \ref{thm:char0} is assumed to hold over $A_c$, we conclude that $M_c \simeq A_c$. Nakayama's lemma lifts this to a surjection $B \to M'$, which yields a surjection $A \to M$. As $A$ is a domain and $M$ is torsion free, we get $A \simeq M$, as desired.
\end{proof}

\subsection{Examples}
\label{subsec:ex}

We now give examples illustrating the necessity of the depth assumption in Theorem \ref{thm:char0} as well as the occurrence of $p$-torsion in Theorem \ref{mainthm:charp}. We begin with an example of the non-injectivity of the restriction map for coherent cohomology; this leads to the desired examples via the exponential.

\begin{example}
	\label{ex:segrecurvepn}
	Fix a canonically embedded smooth projective curve $C$ of genus $g > 1$ over a field $k$. Let $L = \calO_{\P^n}(1) \boxtimes K_C$ be the displayed line bundle on $\P^n \times C$ (for $n > 0$), and let $\V(L^{-1}) \to \P^n \times C$ be its total space. Set $(X,x)$ be the affine cone over $\P^n \times C$ with respect $L$, i.e., $X = \Spec(A)$ where $A := \Gamma(\V(L^{-1}), \calO_{\V(L^{-1})}) = \oplus_{i \geq 0} H^0(\P^n \times C, L^i)$, $x$ is the origin, and let $V = X - \{x\} \subset X$ be the punctured cone; note that $L$ is very ample and $A$ is normal. The affinization map $\V(L^{-1}) \to X$ is the contraction of the $0$ section of $\V(L^{-1})$, so we can view $V$ as the complement of the zero section in $\V(L^{-1})$. In particular, the Kunneth formula shows
	\[ H^0(V,\calO_V) = H^0(X,\calO_X) \simeq \oplus_{i \geq 0} H^0(\P^n,\calO_{\P^n}(i)) \otimes H^0(C,K_C^{\otimes i}) \]
	and
	\[ H^1(V,\calO_V) = \oplus_{i \in \Z} H^1(\P^n \times C, L^i) \simeq \Big(H^0(\P^n,\calO_{\P^n}) \otimes H^1(C,\calO_C)\Big) \oplus \Big(H^0(\P^n,\calO_{\P^n}(1)) \otimes H^1(C,K_C)\Big),\]
	with the evident $H^0(V,\calO_V)$-module structure. Pick non-zero sections $s_1 \in H^0(\P^n,\calO_{\P^n}(1))$ and $s_2 \in H^0(C,K_C)$, and set $f = s_1 \otimes s_2 \in A$.  We will show that multiplication by $f$ on $H^1(V,\calO_V)$ has non-zero image. First, note that $s_2$ defines a map $\calO_C \to K_C$ that induces a surjective non-zero map $H^1(C,\calO_C) \to H^1(C,K_C)$. Since $s_1$ induces an injective map $H^0(\P^n,\calO_{\P^n}) \to H^0(\P^n,\calO_{\P^n}(1))$, it follows $f = s_1 \otimes s_2$ induces a non-zero map 
	\[ H^0(\P^n,\calO_{\P^n}) \otimes H^1(C,\calO_C) \to H^0(\P^n,\calO_{\P^n}(1)) \otimes H^1(C,K_C), \]
	and hence a non-zero endomorphism of $H^1(V,\calO_V)$ by the description above. In particular, if we set $V_0 = V \cap \Spec(A/f) \subset V$, then the map $H^1(V,\calO_V) \to H^1(V_0,\calO_{V_0})$ is not injective. The same calculation is valid after replacing $X$ with its completion $Y$ at $x$, and $V$ and $V_0$ with their preimages $U$ and $U_0$ respectively in $Y$ (as $H^1(V,\calO_V) \simeq H^1(U,\calO_U)$, and similarly for $V_0$). Finally, since $H^1(V,\calO_V)[f] \neq 0$,  the inclusion $A/f \hookrightarrow H^0(V_0,\calO_{V_0})$ is not surjective, so $\depth_x(A/f) = 1$; this reasoning also shows $\depth_x(A/g) = 1$ for any $0 \neq g \in A$ vanishing at $x$.
\end{example}

\begin{remark}
	The construction and conclusion of Example \ref{ex:segrecurvepn} works over any normal ring $k$, and specialises to the desired conclusion over the fibres as long as the sections $s_i$ are chosen to be non-zero in every fibre.
\end{remark}

Via the exponential, we obtain an example illustrating the depth condition in Theorem \ref{thm:char0}:

\begin{example}
	\label{ex:expsegre}
	Consider Example \ref{ex:segrecurvepn} over $k = \C$. The exponential sequence shows $\Pic(V^\an) \to \Pic(V_0^\an)$ is not injective as $H^1(V^\an,\Z)$ is countable. One then also has non-injectivity of $\Pic(W) \to \Pic(W_0)$, where $W$ is any link of $x \in X^\an$, i.e., $W = \overline{W} - \{x\}$ for a small contractible Stein analytic neighbourhood $\overline{W}$ of $x$ in $X^\an$; this is because $H^1(V^\an,\Z) \simeq H^1(W,\Z)$ (as both sides are homotopy equivalent to the circle bundle over $\P^n \times C$ defined by $L^{-1}$), and $H^1(V^\an,\calO_{V^\an}) \simeq H^1(W,\calO_W)$ (by excision and Cartan's Theorem B).  By \cite[Theorem 5]{SiuExtendSheaf}, since any such $\overline{W}$ is normal of dimension $\geq 3$,  we may identify $\Pic(W)$ with isomorphism classes of analytic coherent $S_2$ sheaves on $\overline{W}$ free of rank $1$ over $W$. Nakayama then shows non-injectivity of $\Pic(U) \to \Pic(U_0)$. 
\end{example}

\begin{remark}
	The (punctured) local scheme of Example \ref{ex:expsegre} is not essentially of finite type over $k$, but rather the (punctured) completion of such a scheme;  an essentially finitely presented example can be obtained via Artin approximation. Note that {\em some} approximation is necessary to algebraically detect the analytic line bundles from Example \ref{ex:expsegre} since $\Pic(V) = \Pic(C \times \P^2)/\Z \cdot L$ is smaller than $\Pic(V^\an)$.
\end{remark}

Reducing modulo $p$ (suitably) shows that the map of Theorem \ref{mainthm:charp} often has a non-trivial $p$-torsion kernel:

\begin{example}
	\label{ex:ptorsionpossible}
	Consider Example \ref{ex:segrecurvepn} over $k = \Z[1/N]$ for $n \geq 3$, and suitable choices of $N$, $C$, $s_1$, and $s_2$. Let $B$ be the blowup of $Y$ at $x$; this may be viewed as the base change to $Y$ of the contraction $\V(L^{-1}) \to X$. Write $\widehat{B}$ for the formal completion of $B$ along $i:\P^n \times C  \hookrightarrow B$ (coming from the $0$ section), and let $I \subset \calO_B$ denote the ideal defining $i$, so $i^*(I) \simeq L$. Using formal GAGA for $B \to Y$, one can check that there is an exact sequence
	\[ 1 \to H^1(\widehat{B},1 + I) \to \Pic(B) \to \Pic(\P^n \times C) \to 1 \]
	with a canonical splitting provided by the composite projection $B \to \V(L^{-1}) \to \P^n \times C$. As $n \geq 3$, using Kunneth, one computes
	\begin{equation}
		\label{eq:truncexp}
		H^1(\widehat{B},1+I) \stackrel{\can}{\simeq} H^1(\widehat{B},(1+I)/(1 + I^2)) \stackrel{\log}{\simeq} H^1(\widehat{B},I/I^2) \simeq H^1(\P^n \times C,L),
	\end{equation}
	which, again thanks to Kunneth, gives an exact sequence
	\[ 1 \to \Big(H^0(\P^n,\calO_{\P^n}(1)) \otimes H^1(C,K_C)\Big) \to \Pic(B) \to \Pic(\P^n \times C) \to 1.\]
	The restriction map $\Pic(B) \to \Pic(U)$ has kernel $\Z \cdot L \subset \Pic(\P^n \times C) \subset \Pic(B)$, where the last inclusion comes from the splitting. Thus, there is an injective map
	\[ \Big(H^0(\P^n,\calO_{\P^n}(1)) \otimes H^1(C,K_C)\Big) \hookrightarrow \Pic(U). \]
	We leave it to the reader to check that this map coincides with the one coming from the exponential when specialising to $k = \C$. In particular, after replacing everything in sight with its base change along $k \to \F_p$ for suitable $p$,  we see that $\Pic(U) \to \Pic(U_0)$ has a non-zero kernel; note that, as predicted by Theorem \ref{mainthm:charp}, this kernel is visibly $p$-torsion.
\end{example}

\section{A vector bundle analogue}
\label{sec:bundlethm}

Our goal is to prove the following vector bundle analogue of Theorem \ref{mainthm:charp}:

\begin{theorem}
	\label{thm:bundles}
	Let $X$ be a normal projective variety of dimension $\geq 3$ over an algebraically closed field $k$ of characteristic $p > 0$. If $E \in \Vect(X)$ is trivial over an ample divisor, then $E$ is trivialised by a torsor for a finite connected $k$-group scheme. In particular, $(\Frob_X^e)^* E \simeq \calO_X^{\oplus r}$ for $e \gg 0$.
\end{theorem}

Our approach to Theorem \ref{thm:bundles} is the same as that to Theorem \ref{mainthm:charp}. However, it does not seem straightforward to deduce the former from the latter, so we redo the relevant arguments in a slightly different setting. For the rest of this section, we adopt the following notation:

\begin{notation}
Fix a normal projective variety $X$ of dimension $d$ over an algebraically closed field $k \supset \F_p$, and an ample divisor $H \subset X$.   Let $\overline{X}$ be a fixed absolute integral closure of $X$. For any geometric object $F$ over $X$, write $\overline{F}$ for its pullback to $\overline{X}$. For any $X$-scheme $Y$, we write $Y_n$ for the $n$-th infinitesimal neighbourhood of the inverse image of $H$, and $\widehat{Y}$ for the formal completion of $Y$ along $Y_0$. For $K \in D(\calO_Y)$, write $\widehat{K} \simeq \R\lim(K \otimes_{\calO_Y} \calO_{Y_n})$, viewed as an object on $\widehat{Y}$.  Finally, we use $\underline{\Vect}(Y)$ to denote the {\em groupoid} of vector bundles on $Y$.
\end{notation}

The basic vanishing result that will be used is:

\begin{proposition}
	\label{prop:bundlevanishing}
	For  $E \in \Vect(\overline{X})$, $i < d$ and $n \gg 0$, we have $H^i(\overline{X},E(-n\overline{H})) = 0$.
\end{proposition}
\begin{proof}
	If $E$ is a finite direct sum of twists of $\calO_{\overline{X}}$ by $\overline{H}$, then the claim follows from \cite{HHBigCM}. For the general case, fix a sufficiently large integer $N$. Then the standard construction of free resolutions (applied to the dual of $E$ at some finite level) shows that one can find an exact triangle $E \to P \to Q$ in $D^{\geq 0}(\calO_{\overline{X}})$ such that
	\begin{enumerate}
		\item $P = \Big(P^0 \to P^1 \to \dots \to P^N\Big)$ with $P^i$ a finite direct sum of twists of $\calO_{\overline{X}}$ (in cohomological degree $i$).
		\item $Q$ lies in $D^{\geq N}(\calO_{\overline{X}})$.
	\end{enumerate}
	Then (2) shows that $H^i(\overline{X},E(-n\overline{H})) \simeq H^i(\overline{X},P(-n\overline{H}))$ for $i < d$ and any $n$. By (1), each $H^i(\overline{X},P(-n\overline{H}))$ admits a {\em finite} filtration with graded pieces being subquotients of $H^{i-j}(\overline{X},P^j(-n\overline{H}))$. Each of these subquotients vanishes for $i < d$ and $n \gg 0$. The desired conclusion follows as the filtration is finite.
\end{proof}

We can now algebraise some cohomology groups:

\begin{lemma}
	\label{lem:formalcohomologyvectorbundle}
Assume $d \geq 2$. For any $E \in \Vect(\overline{X})$, we have $H^i(\overline{X},E) \simeq H^i(\widehat{\overline{X}},\widehat{E})$ for $i < d-1$. The analogous claim for $i=0$ is also valid on $X$.
\end{lemma}
\begin{proof}
	We first show the claim for $\overline{X}$. The projective system of exact sequences $1 \to E(-n\overline{H}) \to E \to E|_{\overline{X}_{n-1}} \to 1$ gives a triangle
	\[ \R\lim \R\Gamma(\overline{X},E(-n\overline{H})) \to \R\Gamma(\overline{X},E) \stackrel{a}{\to} \R\Gamma(\widehat{\overline{X}},\widehat{E}). \]
	The left hand side lies in $D^{[d,d+1]}(k)$ by Proposition \ref{prop:bundlevanishing}, so $H^i(a)$ is an isomorphism for $i < d-1$. For $X$, the same argument applies once we observe that $H^0(X,E(-nH)) = 0$ for $n \gg 0$ by ampleness as $d \geq 1$, and that $H^1(X,E(-nH)) = 0$ for $n \gg 0$ by the Lemma of Enriques-Severi-Zariski as $d \geq 2$.
\end{proof}

Passage to formal completions of ample divisors faithfully reflects the geometry of bundles:

\begin{lemma}
	\label{lem:completionfullyfaithfulbundle}
	Assume $d \geq 2$. The functor $\underline{\Vect}(\overline{X}) \to \underline{\Vect}(\widehat{\overline{X}})$ is fully faithful, and similarly on $X$.
\end{lemma}
\begin{proof}
	Lemma \ref{lem:formalcohomologyvectorbundle} shows that $\Hom(E,F) \simeq \Hom(\widehat{E},\widehat{F})$ for $E,F \in \Vect(\overline{X})$ (or $\Vect(X)$). It now suffices to check that if $f:E \to F$ induces an isomorphism $\widehat{f}:\widehat{E} \to \widehat{F}$, then $f$ is itself an isomorphism. By taking determinants, we may assume $E$ and $F$ are line bundles. As the reduction $f_0:E_0 \to F_0$ is an isomorphism, the support of $\coker(f)$ is a divisor that does not intersect $\overline{H}$, contradicting ampleness.
\end{proof}

We obtain a Lefschetz-type result for $\pi_1$:

\begin{corollary}
	\label{cor:lefschetzpi1bundle}
	Assume $d \geq 2$. The map $\pi_1(X_0) \to \pi_1(X)$  is surjective.
\end{corollary}
\begin{proof}
	We first observe that $X_0$ is connected by the Lemma of Enriques-Severi-Zariski, so the notation is unambiguous. As $\pi_1(X_0) \simeq \pi_1(X_n) \simeq \pi_1(\widehat{X})$, it suffices to observe: for any finite \'etale $\calO_X$-algebra $\calA$, the natural map $H^0(X,\calA) \to H^0(\widehat{X},\widehat{\calA})$  is an isomorphism of algebras by Lemma \ref{lem:formalcohomologyvectorbundle}, and hence identifies idempotents.  
\end{proof}

Using the vanishing of cohomology on $\overline{X}$, deformations of the trivial bundle on $\overline{X}_0$ are easy to classify:

\begin{lemma}
	\label{lem:restrictionfullyfaithfulbundle}
	Assume $d \geq 3$. The fibre over the trivial bundle of $\underline{\Vect}(\widehat{\overline{X}}) \to \underline{\Vect}(\overline{X}_0)$ is contractible.
\end{lemma}
\begin{proof}
	Let $E = \calO_{\overline{X}}^{\oplus r}$. It suffices to show that the fibre $F_n$ over $E_{n-1}$ of $\underline{\Vect}(\overline{X}_{n}) \to \underline{\Vect}(\overline{X}_{n-1})$ is contractible for $n \geq 1$. One has $\pi_0(F_n) = H^1(\overline{X}_0, \underline{\End}(E_0)(-n\overline{H})) \simeq H^1(\overline{X}_0,\calO_{\overline{X}_0}(-n\overline{H}))^{\oplus r^2}$. This group vanishes by Proposition \ref{prop:bundlevanishing} and the exact sequence 
	\[ 1 \to \calO_{\overline{X}}(-(n+1)\overline{H}) \to \calO_{\overline{X}}(-n\overline{H}) \to \calO_{\overline{X}_0}(-n\overline{H}) \to 1\]
	as $d \geq 3$. A similar argument shows that $\pi_1(F_n) = \ker(H^0(\overline{X}_0,\underline{\End}(E_0)(-n\overline{H}))) = 0$, which proves the claim.
\end{proof}

We can now prove the promised result:

\begin{proof}[Proof of Theorem \ref{thm:bundles}]
	Fix an $E \in \Vect(X)$ with $E|_H \simeq \calO_H^{\oplus r}$. Then lemmas \ref{lem:completionfullyfaithfulbundle} and \ref{lem:restrictionfullyfaithfulbundle} show that $\overline{E}$ is the trivial bundle over $\overline{X}$. Hence, there is a finite cover of $X$ trivialising $E$. By \cite{MehtaAnteipi1}, there is a finite $k$-group scheme $G$ such that $E$ is trivialized by a $G$-torsor over $X$. Using Corollary \ref{cor:lefschetzpi1bundle} and the connected-\'etale sequence for $G$, we may choose $G$ to be connected, proving half the claim. The last part follows from the observation that any finite surjective purely inseparable map $Y \to X$ is dominated by a power of Frobenius on $X$.
\end{proof}

We end by noting that the {\em proof} of Corollary \ref{cor:lefschetzpi1bundle}, Fujita vanishing \cite[Theorem 10]{FujitaVanishing}, and representability results for Picard functors (see \cite{KleimanPicardScheme}) can be used to prove the following Lefschetz-type result for base-point free big divisors on normal varieties. We thank Brian Lehmann for bringing this question to our attention.

\begin{theorem}
	\label{thm:sabig}
Let $X$ be a normal projective variety of dimension $\geq 2$ over a field $k$, and fix a Cartier divisor $D \subset X$ such that $\calO(D)$ is semiample and big. Then the restriction map $\Pic^\tau(X) \to \Pic^\tau(D)$ is:
	\begin{enumerate}
		\item injective if $k$ has characteristic $0$.
		\item injective up to a finite and $p^\infty$-torsion kernel if $k$ has characteristic $p > 0$.
	\end{enumerate}
\end{theorem}

In \cite{RavindraSrinivasGrothendieckLefschetz}, one finds a stronger result with stronger assumptions: they completely describe the kernel and cokernel of $\Pic(X) \to \Pic(D)$ when  $X$ is a smooth projective variety in characteristic $0$, and $D$ is general in its linear system.

\bibliography{my}
\bibliographystyle{amsalpha}
\end{document}